\documentclass[copyright,creativecommons]{eptcs}
\providecommand{\event}{E. Formenti (Ed.): AUTOMATA \& JAC 2012}

\usepackage{amsthm,amsmath,amscd,amssymb}
\usepackage{comment}
\usepackage{tikz}
\usetikzlibrary{arrows}

\newtheorem{lemma}{Lemma}

\newtheorem{corollary}{Corollary}
\newtheorem{proposition}{Proposition}
\newtheorem{example}{Example}
\newtheorem{theorem}{Theorem}

\newtheorem{question}{Question}

\newcommand{\Z}{\mathbb{Z}}
\newcommand{\N}{\mathbb{N}}
\newcommand{\lang}{\mathcal{L}}
\newcommand{\supp}{\mbox{supp}}
\newcommand{\tower}{\mbox{tower}}

\title{On Nilpotency and Asymptotic Nilpotency of Cellular Automata\thanks{Research supported by the Academy of Finland Grant 131558}}
\def\titlerunning{On Nilpotency and Asymptotic Nilpotency of Cellular Automata}

\author{
	Ville Salo
		\institute{TUCS -- Turku Center for Computer Science, \\
		University of Turku, Finland}
		\email{vosalo@utu.fi}
}
\def\authorrunning{Ville Salo}

\begin{document}
\maketitle

\begin{abstract}
We prove a conjecture from [Guillon-Richard '08] by showing that cellular automata that eventually fix all cells to a fixed symbol $0$ are nilpotent on $S^{\Z^d}$ for all $d$. We also briefly discuss nilpotency on other subshifts, and show that weak nilpotency implies nilpotency in all subshifts and all dimensions, since we do not know a published reference for this.
\end{abstract}

\section{Introduction}

One of the most interesting aspects in the theory of cellular automata is the study of different types of nilpotency, that is, different ways in which a cellular automaton can force a particular symbol (usually called $0$) to appear frequently in all its spacetime diagrams. The simplest such notion, called simply `nilpotency', is that the cellular automaton $c$ maps every configuration to a uniform configuration $\ldots 000 \ldots$, on which it behaves as the identity, in a uniformly bounded number of steps, that is, $c^n$ is a constant map for some $n$. The notion of `weak nilpotency', where for all $x \in X$, we have $c^n(x) = \ldots 000 \ldots$ for some $n$, is equivalent to nilpotency in all subshifts $X \subset S^{\Z^d}$. We don't know a published reference for this, and give a proof in Proposition~\ref{prop:WeakIsStrong}, slightly strengthening Proposition 2 of \cite{GuGa10}. There are at least two ways to define variants of nilpotency using measure theory. A notion called `ergodicity' is discussed in \cite{BuMaMa10}, where it is also shown to be equivalent to nilpotency, while another notion called `unique ergodicity' is shown to be strictly weaker in \cite{To12}. Unique ergodicity refers to the existence of a unique invariant measure for the cellular automaton, and ergodicity refers to unique ergodicity with the additional assumption that every measure converges weakly to the unique invariant measure in the orbit of the cellular automaton.

A particularly nice variant of nilpotency is the `asymptotic nilpotency' defined and investigated in at least \cite{GuGa08} (where asymptotically nilpotent cellular automata are called `cellular automata with a nilpotent trace') and \cite{GuGa10}. This notion is shown to be equivalent to nilpotency on one-dimensional full shifts in \cite{GuGa08}, and on one-dimensional transitive SFTs in \cite{GuGa10}, but interestingly the proofs are not trivial, unlike for most of the notions listed above which coincide with nilpotency. It was left open in \cite{GuGa08} whether asymptotic nilpotency implies nilpotency in all dimensions, that is, on the groups $\Z^d$ for arbitrary $d$, and the proof given in \cite{GuGa08} does not generalize above $d = 1$ as such. The problem is restated as open in \cite{GuGa10}. In this article, we prove that asymptotic nilpotency indeed implies nilpotency in all dimensions in Theorem~\ref{thm:dDCase} by reducing to the one-dimensional case, and ask whether this holds in all groups.

The study of nilpotency properties of cellular automata belongs to the more general study of asymptotic behavior of cellular automata, where emphasis is usually put on the limit set of a cellular automaton, see \cite{Ku09} for a survey on this topic. More generally, the study of cellular automata is subsumed by the study of multidimensional SFTs, in that a $d$-dimensional cellular automaton can be thought of as a $(d + 1)$-dimensional SFT deterministic in one dimension. See \cite{PaSc10} for a thorough investigation into the `limit behavior' of these systems.

We also slightly extend the result of \cite{GuGa10} that asymptotic nilpotency implies nilpotency on one-dimensional transitive SFTs by removing the assumption of transitivity, and using this observation obtain some more cases in which asymptotic nilpotency implies nilpotency on multidimensional SFTs in Theorem~\ref{thm:FinitesDense}. We also give two simple examples of one-dimensional sofic shifts where asymptotic nilpotency does not imply nilpotency. Finally, we ask whether \emph{any} two-dimensional SFT can admit an asymptotically nilpotent but not nilpotent cellular automaton.

\section{Definitions and Initial Observations}

Let $S$ be a finite set called the \emph{alphabet}, whose elements we call \emph{symbols}. The set $S^{\Z^d}$ with the product topology induced by the discrete topology of $S$ is called the \emph{$d$-dimensional full shift}. We call the elements $x \in S^{\Z^d}$ \emph{configurations}, and for $s \in S$, we write $s^{\Z^d}$ for the \emph{all-$s$ configuration} $x$ such that $x_v = s$ for all $v \in \Z^d$. We denote by $\sigma^v : S^{\Z^d} \to S^{\Z^d}$ the \emph{shift action} defined by $\sigma^v(x)_u = x_{u + v}$. The topologically closed sets which are invariant under $\sigma^v$ for all $v \in \Z^d$ are called \emph{subshifts}. Alternatively, they can be defined by a set of forbidden patterns \cite{LiMa95}, and when this set can be taken to be finite, we say $X$ is a subshift of finite type, or an \emph{SFT}. For a subshift $X \subset S^{\Z^d}$, a continuous function $c : X \to X$ is called a \emph{cellular automaton} if it commutes with $\sigma^v$ for all $v \in \Z^d$. Of course, the shift actions $\sigma^v$ are themselves cellular automata. Cellular automata have a combinatorial description as functions induced by a local rule \cite{He69}. That is, if $c$ is a cellular automaton, there is a finite set of vectors $V$ such that $c(x)_u$ only depends on the pattern $x_{u+V}$. The quantity $r = \max_{v \in V} |v|$ is called the \emph{radius} of $c$, where $|\cdot|$ is the norm $|v| = \max_i |v_i|$.

For a one-dimensional subshift $X$, we denote by $\lang(X)$ the set of words that appear in configurations of $X$, called the \emph{language} of $X$. We denote by $\lang^{-1}(L)$ the topological closure of the set of configurations whose subwords are words in $\mbox{fact}(L)$, where $\mbox{fact}(L)$ is the closure of $L$ under taking subwords. The set $\lang^{-1}(L)$ is a subshift for all languages $L$, and when $\lang(X) = L$ for a subshift $X$, we have $\lang^{-1}(L) = X$. We say a one-dimensional subshift $X$ is \emph{sofic} when $\lang(X)$ is a regular language, and in general, a $d$-dimensional subshift is called sofic when it is the projection of an SFT of dimension $d$ by a pointwise symbol-to-symbol map. When $d = 1$, these definitions coincide \cite{LiMa95}.

A one-dimensional subshift $X$ is said to be \emph{transitive} if
\[ u, v \in \lang(X) \implies (\exists w)(uwv \in \lang(X)), \]
and $X$ is said to be \emph{mixing} if
\[ u, v \in \lang(X) \implies (\exists m)(\forall n \geq m)(\exists w)(|w| = n \wedge uwv \in \lang(X)), \]
that is, transitivity implies that any two words that occur in the subshift can be glued together with some word $w$, and mixing implies that the length can be chosen freely provided it is sufficiently large. We will also need the decomposition of a one-dimensional sofic shift into its transitive components in Corollary~\ref{cor:SFTCase}, and refer to \cite{LiMa95} for the definitions and proofs.

Let $c : X \to X$ be a cellular automaton on a subshift $X \subset S^{\Z^d}$. We say $c$ is \emph{$0$-nilpotent} if
\[ (\exists m)(\forall n \geq m)(\forall x \in X)(\forall v \in \Z^d)(c^n(x)_v = 0). \]
The \emph{limit set} of $c$ is the set of configurations $x$ for which there exists an infinitely long preimage chain $(x_i)_{i \in \N}$ such that $x_0 = x$ and $c(x_{i+1}) = x_i$. It is well-known that $c$ is $s$-nilpotent for some $s \in S$ if and only if its limit set is a singleton \cite{CuPaYu89}. We say $c$ is \emph{weakly $0$-nilpotent} if
\[ (\forall x \in X)(\exists m)(\forall n \geq m)(\forall v \in \Z^d)(c^n(x)_v = 0), \]
and we say $c$ is \emph{asymptotically $0$-nilpotent} if
\[ (\forall x \in X)(\forall v \in \Z^d)(\exists m)(\forall n \geq m)(c^n(x)_v = 0). \]
It is clear from these formulas that nilpotency implies weak nilpotency, which in turn implies asymptotic nilpotency. In this article, we prove that the first implication is always an equivalence, and the second is an equivalence at least when $X = S^{\Z^d}$ (see Theorem~\ref{thm:FinitesDense} for the exact cases we are able to prove). A symbol $s$ is said to be \emph{quiescent} for $c$ if $c(s^{\Z^d}) = s^{\Z^d}$. If $c$ is $0$-nilpotent, weakly $0$-nilpotent or asymptotically $0$-nilpotent, then $0$ must be a quiescent state. Usually, the symbol $0$ is clear from the context and is omitted.

A configuration $x$ is said to be \emph{$0$-finite} if it has only finitely many non-$0$ symbols. It is called \emph{$0$-mortal for $c$} if $c^n(x) = 0^{\Z^d}$ for large enough $n$. Note that weak $0$-nilpotency is equivalent to every configuration being $0$-mortal. Given a configuration $x$, its \emph{trace} is the one-way infinite word $(c^n(x)_{\vec{0}})_{n \in \N}$, where $\vec{0} = (0, \ldots, 0) \in \Z^d$, and its \emph{trace support} is the set $\{n \in \N \;|\; c^n(x)_{\vec{0}} \neq 0\}$. Asymptotic $0$-nilpotency is then equivalent to all configurations having a finite trace support. For $x \in S^{\Z^d}$, we write $\supp_0(x) = \{v \in \Z^d \;|\; x_v \neq 0\}$, called the \emph{support} of $x$, so a $0$-finite configuration is just a configuration with finite $0$-support. For $x, y \in S^{\Z^d}$ with $\supp_0(x) \cap \supp_0(y) = \emptyset$, we define
\[ (x +_0 y)_v = \left\{\begin{array}{ll}
x_v, & \mbox{if } x_v \neq 0, \\
y_v, & \mbox{if } y_v \neq 0, \\
0, & \mbox{otherwise.}
\end{array}\right. \]
Again, in all these definitions, both $0$ and $c$ are omitted if they are clear from context.

For a subshift $X \subset S^{\Z^d}$ and $v \in \Z^d$, we denote by $X(v)$ the subshift $\{x \in X \;|\; \sigma^v(x) = x\}$. When $v^j = (0, \ldots, 0, 1, 0, \ldots, 0)$ is a standard basis vector and $1 \leq p \in \N$, we define the map $\phi_{j, p} : X(p v^j) \to (S^p)^{\Z^{d-1}}$ by $\phi_{j, p}(x)_u = x_{u(0)} x_{u(1)} \cdots x_{u(p-1)}$ where $u(i) = (u_1, \ldots, u_{j-1}, i, u_j, \ldots u_{d-1})$. Clearly, $\phi_{j, p}$ is a homeomorphism between $X(p v^j)$ and $Y = \phi_{j, p}(X(p v^j))$, and we obtain a $(d - 1)$-dimensional action $c_{j, p} : Y \to Y$ by defining $c_{j, p}(\phi_{j, p}(x)) = \phi_{j, p}(c(x))$. This is well-defined because $\phi_{j, p}$ is bijective, and it is continuous and shift-commuting because $\phi_{j, p}, \phi_{j, p}^{-1}$ and $c$ are, so $c_{j, p}$ is a cellular automaton (although some small technical care needs to be taken in showing this due to the coordinate shift at $j$). Moreover, it is easy to see that $c_{j, p}$ is nilpotent (asymptotically nilpotent) if and only if $c$ is nilpotent (asymptotically nilpotent) on $X(p v^j)$.

For a vector $v \in \Z^d$ and $k \in \N$, we define $B_k(v) = \{ u \;|\; |u - v| \leq k \}$ where again $|\cdot|$ is the norm $|v| = \max_i |v_i|$, and for a set of vectors $V \subset \Z^d$ and $k \in \N$, we define $B_k(V) = \bigcup_{v \in V} B_k(v)$. For $v \in \Z^d$ and $j, k \in \N$, we define 
\[ \tower(j, k, v) = B_k(\{u \in \Z^d \;|\; \forall i \neq j: u_i = v_i\}) \subset \Z^d. \]
This is a tower of width $k$ around the vector $v$, extending in the directions $v^j$ and $-v^j$. Finally, for a set of vectors $V \subset \Z^d$ and $j, k \in \N$, we analogously define
\[ \tower(j, k, V) = \bigcup_{v \in V} \tower(j, k, v). \]
See Figure~\ref{fig:BallAndTower} for an illustration of these concepts.

\begin{figure}[ht]
\begin{center}
\begin{tikzpicture}[scale = 0.2]

\fill[draw=black,color=black!20] (-1,-7) rectangle (18,24);

\fill[draw=black,color=black] (5,8) rectangle (6,9);
\fill[draw=black,color=black] (6,7) rectangle (7,8);
\fill[draw=black,color=black] (6,8) rectangle (7,9);
\fill[draw=black,color=black] (7,5) rectangle (8,6);
\fill[draw=black,color=black] (7,8) rectangle (8,9);
\fill[draw=black,color=black] (7,9) rectangle (8,10);
\fill[draw=black,color=black] (8,4) rectangle (9,5);
\fill[draw=black,color=black] (8,5) rectangle (9,6);
\fill[draw=black,color=black] (8,8) rectangle (9,9);
\fill[draw=black,color=black] (8,11) rectangle (9,12);
\fill[draw=black,color=black] (9,5) rectangle (10,6);
\fill[draw=black,color=black] (9,6) rectangle (10,7);
\fill[draw=black,color=black] (9,8) rectangle (10,9);
\fill[draw=black,color=black] (9,9) rectangle (10,10);
\fill[draw=black,color=black] (9,10) rectangle (10,11);
\fill[draw=black,color=black] (9,11) rectangle (10,12);
\fill[draw=black,color=black] (10,8) rectangle (11,9);
\fill[draw=black,color=black] (11,8) rectangle (12,9);
\fill[draw=black,color=black] (11,9) rectangle (12,10);

\fill[draw=black,color=black!35] (7,3) rectangle (8,4);
\fill[draw=black,color=black!35] (6,9) rectangle (7,10);
\fill[draw=black,color=black!35] (11,11) rectangle (12,12);
\fill[draw=black,color=black!35] (7,12) rectangle (8,13);
\fill[draw=black,color=black!35] (12,12) rectangle (13,13);
\fill[draw=black,color=black!35] (3,7) rectangle (4,8);
\fill[draw=black,color=black!35] (2,5) rectangle (3,6);
\fill[draw=black,color=black!35] (5,5) rectangle (6,6);
\fill[draw=black,color=black!35] (11,5) rectangle (12,6);
\fill[draw=black,color=black!35] (10,7) rectangle (11,8);
\fill[draw=black,color=black!35] (7,6) rectangle (8,7);
\fill[draw=black,color=black!35] (6,10) rectangle (7,11);
\fill[draw=black,color=black!35] (12,6) rectangle (13,7);
\fill[draw=black,color=black!35] (13,7) rectangle (14,8);
\fill[draw=black,color=black!35] (4,10) rectangle (5,11);
\fill[draw=black,color=black!35] (2,6) rectangle (3,7);
\fill[draw=black,color=black!35] (9,14) rectangle (10,15);
\fill[draw=black,color=black!35] (8,2) rectangle (9,3);
\fill[draw=black,color=black!35] (5,11) rectangle (6,12);
\fill[draw=black,color=black!35] (4,5) rectangle (5,6);
\fill[draw=black,color=black!35] (10,13) rectangle (11,14);
\fill[draw=black,color=black!35] (9,3) rectangle (10,4);
\fill[draw=black,color=black!35] (12,11) rectangle (13,12);
\fill[draw=black,color=black!35] (13,10) rectangle (14,11);
\fill[draw=black,color=black!35] (8,12) rectangle (9,13);
\fill[draw=black,color=black!35] (2,11) rectangle (3,12);
\fill[draw=black,color=black!35] (5,14) rectangle (6,15);
\fill[draw=black,color=black!35] (10,14) rectangle (11,15);
\fill[draw=black,color=black!35] (6,13) rectangle (7,14);
\fill[draw=black,color=black!35] (14,8) rectangle (15,9);
\fill[draw=black,color=black!35] (12,8) rectangle (13,9);
\fill[draw=black,color=black!35] (3,11) rectangle (4,12);
\fill[draw=black,color=black!35] (8,9) rectangle (9,10);
\fill[draw=black,color=black!35] (4,12) rectangle (5,13);
\fill[draw=black,color=black!35] (9,4) rectangle (10,5);
\fill[draw=black,color=black!35] (5,1) rectangle (6,2);
\fill[draw=black,color=black!35] (10,3) rectangle (11,4);
\fill[draw=black,color=black!35] (7,2) rectangle (8,3);
\fill[draw=black,color=black!35] (6,14) rectangle (7,15);
\fill[draw=black,color=black!35] (12,2) rectangle (13,3);
\fill[draw=black,color=black!35] (11,10) rectangle (12,11);
\fill[draw=black,color=black!35] (14,5) rectangle (15,6);
\fill[draw=black,color=black!35] (12,13) rectangle (13,14);
\fill[draw=black,color=black!35] (3,6) rectangle (4,7);
\fill[draw=black,color=black!35] (8,6) rectangle (9,7);
\fill[draw=black,color=black!35] (10,9) rectangle (11,10);
\fill[draw=black,color=black!35] (9,7) rectangle (10,8);
\fill[draw=black,color=black!35] (6,4) rectangle (7,5);
\fill[draw=black,color=black!35] (5,4) rectangle (6,5);
\fill[draw=black,color=black!35] (11,4) rectangle (12,5);
\fill[draw=black,color=black!35] (10,4) rectangle (11,5);
\fill[draw=black,color=black!35] (7,1) rectangle (8,2);
\fill[draw=black,color=black!35] (6,11) rectangle (7,12);
\fill[draw=black,color=black!35] (12,7) rectangle (13,8);
\fill[draw=black,color=black!35] (14,6) rectangle (15,7);
\fill[draw=black,color=black!35] (13,6) rectangle (14,7);
\fill[draw=black,color=black!35] (4,11) rectangle (5,12);
\fill[draw=black,color=black!35] (3,5) rectangle (4,6);
\fill[draw=black,color=black!35] (2,7) rectangle (3,8);
\fill[draw=black,color=black!35] (9,13) rectangle (10,14);
\fill[draw=black,color=black!35] (8,3) rectangle (9,4);
\fill[draw=black,color=black!35] (5,10) rectangle (6,11);
\fill[draw=black,color=black!35] (4,6) rectangle (5,7);
\fill[draw=black,color=black!35] (10,10) rectangle (11,11);
\fill[draw=black,color=black!35] (9,2) rectangle (10,3);
\fill[draw=black,color=black!35] (6,1) rectangle (7,2);
\fill[draw=black,color=black!35] (5,7) rectangle (6,8);
\fill[draw=black,color=black!35] (11,3) rectangle (12,4);
\fill[draw=black,color=black!35] (7,4) rectangle (8,5);
\fill[draw=black,color=black!35] (14,12) rectangle (15,13);
\fill[draw=black,color=black!35] (12,4) rectangle (13,5);
\fill[draw=black,color=black!35] (13,9) rectangle (14,10);
\fill[draw=black,color=black!35] (8,13) rectangle (9,14);
\fill[draw=black,color=black!35] (4,8) rectangle (5,9);
\fill[draw=black,color=black!35] (2,8) rectangle (3,9);
\fill[draw=black,color=black!35] (5,13) rectangle (6,14);
\fill[draw=black,color=black!35] (6,2) rectangle (7,3);
\fill[draw=black,color=black!35] (11,14) rectangle (12,15);
\fill[draw=black,color=black!35] (7,11) rectangle (8,12);
\fill[draw=black,color=black!35] (14,9) rectangle (15,10);
\fill[draw=black,color=black!35] (12,9) rectangle (13,10);
\fill[draw=black,color=black!35] (13,12) rectangle (14,13);
\fill[draw=black,color=black!35] (3,10) rectangle (4,11);
\fill[draw=black,color=black!35] (8,10) rectangle (9,11);
\fill[draw=black,color=black!35] (12,3) rectangle (13,4);
\fill[draw=black,color=black!35] (11,13) rectangle (12,14);
\fill[draw=black,color=black!35] (7,14) rectangle (8,15);
\fill[draw=black,color=black!35] (14,10) rectangle (15,11);
\fill[draw=black,color=black!35] (12,14) rectangle (13,15);
\fill[draw=black,color=black!35] (3,9) rectangle (4,10);
\fill[draw=black,color=black!35] (8,7) rectangle (9,8);
\fill[draw=black,color=black!35] (4,2) rectangle (5,3);
\fill[draw=black,color=black!35] (6,5) rectangle (7,6);
\fill[draw=black,color=black!35] (5,3) rectangle (6,4);
\fill[draw=black,color=black!35] (11,7) rectangle (12,8);
\fill[draw=black,color=black!35] (10,5) rectangle (11,6);
\fill[draw=black,color=black!35] (7,13) rectangle (8,14);
\fill[draw=black,color=black!35] (14,7) rectangle (15,8);
\fill[draw=black,color=black!35] (13,5) rectangle (14,6);
\fill[draw=black,color=black!35] (3,4) rectangle (4,5);
\fill[draw=black,color=black!35] (9,12) rectangle (10,13);
\fill[draw=black,color=black!35] (5,9) rectangle (6,10);
\fill[draw=black,color=black!35] (4,7) rectangle (5,8);
\fill[draw=black,color=black!35] (10,11) rectangle (11,12);
\fill[draw=black,color=black!35] (9,1) rectangle (10,2);
\fill[draw=black,color=black!35] (6,6) rectangle (7,7);
\fill[draw=black,color=black!35] (5,6) rectangle (6,7);
\fill[draw=black,color=black!35] (11,2) rectangle (12,3);
\fill[draw=black,color=black!35] (10,6) rectangle (11,7);
\fill[draw=black,color=black!35] (7,7) rectangle (8,8);
\fill[draw=black,color=black!35] (12,5) rectangle (13,6);
\fill[draw=black,color=black!35] (13,8) rectangle (14,9);
\fill[draw=black,color=black!35] (8,14) rectangle (9,15);
\fill[draw=black,color=black!35] (4,9) rectangle (5,10);
\fill[draw=black,color=black!35] (2,9) rectangle (3,10);
\fill[draw=black,color=black!35] (8,1) rectangle (9,2);
\fill[draw=black,color=black!35] (5,12) rectangle (6,13);
\fill[draw=black,color=black!35] (4,4) rectangle (5,5);
\fill[draw=black,color=black!35] (10,12) rectangle (11,13);
\fill[draw=black,color=black!35] (6,3) rectangle (7,4);
\fill[draw=black,color=black!35] (11,1) rectangle (12,2);
\fill[draw=black,color=black!35] (7,10) rectangle (8,11);
\fill[draw=black,color=black!35] (12,10) rectangle (13,11);
\fill[draw=black,color=black!35] (13,11) rectangle (14,12);
\fill[draw=black,color=black!35] (2,10) rectangle (3,11);
\fill[draw=black,color=black!35] (10,1) rectangle (11,2);
\fill[draw=black,color=black!35] (6,12) rectangle (7,13);
\fill[draw=black,color=black!35] (11,12) rectangle (12,13);
\fill[draw=black,color=black!35] (14,11) rectangle (15,12);
\fill[draw=black,color=black!35] (3,8) rectangle (4,9);
\fill[draw=black,color=black!35] (4,3) rectangle (5,4);
\fill[draw=black,color=black!35] (5,2) rectangle (6,3);
\fill[draw=black,color=black!35] (11,6) rectangle (12,7);
\fill[draw=black,color=black!35] (10,2) rectangle (11,3);

\draw[black!50] (-7,-7) grid (24,24);

\node[circle,right] (p) at (26, 18.5) {$V = \{\mbox{a finite set of black cells}\}$};
\node[circle,right] (b) at (26, 13.5) {$B_3(V)$};
\node[circle,right] (t) at (26, 8.5)  {$\tower(2, 6, V)$};
\draw[thick] (11.5,8.5) -- (p.west);
\draw[thick] (14.5,7.5) -- (b.west);
\draw[thick] (17.5,6.5) -- (t.west);
\end{tikzpicture}
\end{center}

\caption{A ball of radius $3$ and a vertical tower of width $6$ around a finite set of vectors. The vertical direction is the second axis.}
\label{fig:BallAndTower}
\end{figure}
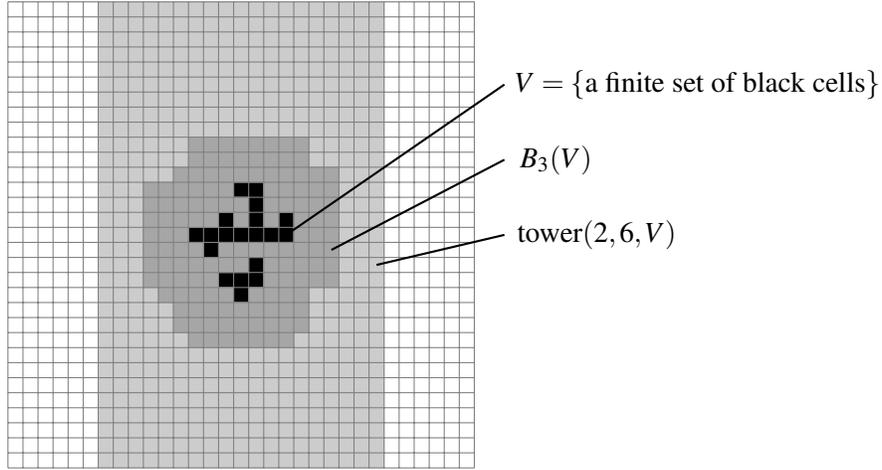

\section{The Results}

First, we show that weak nilpotency implies (and is thus equivalent to) nilpotency in all subshifts $X \subset S^{\Z^d}$. Note that Proposition 2 of \cite{GuGa10} already proves the claim for transitive subshifts on all groups, but it is not clear how to remove the requirement of transitivity. The easiest proof in the transitive case is probably obtained as follows: transitivity guarantees the existence of a configuration $x$ that contains every finite pattern, and if $c^n(x) = 0^{\Z^d}$ then every pattern maps to $0$ in $n$ steps, so if $c$ is weakly nilpotent, it is nilpotent. The general case is not much harder.

\begin{proposition}
\label{prop:WeakIsStrong}
Let $X \subset S^{\Z^d}$ be a subshift. Then a cellular automaton $c$ on $X$ is nilpotent if and only if it is weakly nilpotent.
\end{proposition}

\begin{proof}
Suppose on the contrary that the CA $c : X \to X$ is weakly nilpotent but not nilpotent for some subshift $X \subset S^{\Z^d}$. Then there exists a configuration $x$ in the limit set of $c$ with $x_{\vec{0}} \neq 0$, for the symbol $0$ such that all configurations reach the all-$0$ configuration in finitely many steps. Since $x$ is in the limit set, it has an infinite chain $(x^i)_{i \in \N}$ of preimages. If $r$ is the radius of $c$, then since $0$ is a quiescent state, there must exist a sequence of vectors $(v^i)_{i \in \N}$ such that $|v^i - v^{i+1}| \leq r$ and $\sigma^{v^i}(x^i)_{\vec{0}} \neq 0$ for all $i \in \N$.

Let $y$ be a limit of a converging subsequence of $(\sigma^{v^i}(x^i))_{i \in \N}$ in the product topology of $S^{\Z^d}$. Since $c$ is weakly nilpotent, there exists $n \in \N$ such that $c^n(y) = 0^{\Z^d}$. Let $i \in \N$ be such that
\begin{equation}
\label{eq:XiAndY}
y_{B_{2rn}(\vec{0})} = \sigma^{v^{i+n}}(x^{i+n})_{B_{2rn}(\vec{0})}.
\end{equation}
By definition of the $x^i$ and $v^i$, we have that $c^n(x^{i+n}) = x^i$ and thus
\[ c^n(y)_{B_{rn}(\vec{0})} = \sigma^{v^{i+n}}(x^i)_{B_{rn}(\vec{0})} \]
contains a nonzero symbol, a contradiction, since we assumed $c^n(y) = 0^{\Z^d}$.
\end{proof}

The proof naturally generalizes for cellular automata on subshifts of $S^G$ where $G$ is any group generated by a finite set $A$, by using the distance $d(g, h) = \min\{n \;|\; \exists g_1, \ldots, g_n \in A: g^{-1}h = g_1 \ldots, g_n\}$. What makes Proposition~\ref{prop:WeakIsStrong} particularly interesting is that these notions are \emph{not} equivalent for discrete dynamical systems (compact metrizable spaces paired with a continuous function) in general, and we give a simple counterexample.

\begin{example}
Consider the (compact, metrizable) Alexandroff one-point compactification $\alpha \N$ of the natural numbers with the point $\infty$ at infinity coupled with the action
\[ \phi(n) = \left\{\begin{array}{ll} n-1, & \mbox{if } 0 < n < \infty, \\ \infty, & \mbox{otherwise.} \end{array}\right. \]
It is easy to check that this is a continuous function, so that $(\alpha \N, \phi)$ is a dynamical system. In this system, when $n \neq \infty$ we have that $\phi^{n+1}(n) = \infty$ but $\phi^j(n) \neq \infty$ for $j \leq n$, and $\phi(\infty) = \infty$. This means that $(\alpha \N, \phi)$ is weakly nilpotent but not nilpotent.
\end{example}

Now, we proceed to our main results. We start by repeating Theorem 3 of \cite{GuGa08} and then show how to apply it to prove Conjecture 1 of \cite{GuGa10} in Theorem~\ref{thm:dDCase}.

\begin{theorem}[Theorem 3 of \cite{GuGa08}]
\label{thm:1DCase}
Asymptotic nilpotency implies nilpotency in cellular automata on $S^\Z$.
\end{theorem}

In fact, the following results were essentially already proved in \cite{GuGa08} on full shifts in all dimensions. We will briefly outline the proofs for completeness.

\begin{lemma}
\label{lem:UniformVisits}
If $X \subset S^{\Z^d} \to S^{\Z^d}$ is a subshift and the cellular automaton $c : X \to X$ is asymptotically nilpotent, then
\[ (\forall k)(\exists n)(\forall x \in X)(\exists 0 \leq j \leq n)(\forall v \in B_k(\vec{0}))(c^j(x)_v = 0). \]
\end{lemma}

\begin{proof}
We need to prove that for all $k$, there is a uniform bound $n$ for the first time the cells $B_k(\vec{0})$ are simultaneously zero. But if this were not the case for some $k$, we would find for all $n \in \N$ a configuration $x_n$ where $c^j(x_n)_{B_k(\vec{0})}$ contains a nonzero value for all $0 \leq j \leq n$. By taking a limit point of the sequence $x_n$, we would then obtain a configuration $x$ such that $c^j(x)_{B_k(\vec{0})}$ contains a nonzero value for all $j \in \N$, a contradiction.
\end{proof}

\begin{lemma}
\label{lem:FiniteMortality}
If $X \subset S^{\Z^d}$ is an SFT with dense finite points, the cellular automaton $c : X \to X$ is asymptotically nilpotent, and all finite patterns are mortal for $c$, then $c$ is nilpotent.
\end{lemma}

\begin{proof}
If $c$ is not nilpotent, then for any $k$, there exists a configuration $z_k$ where the trace support contains a number larger than $k$. We may assume the $z_k$ are finite by the assumption that finite points are dense, and they are then automatically mortal. By Lemma~\ref{lem:UniformVisits}, the $z_k$ can further be taken to have no nonzero values in the coordinates $B_k(\vec{0})$. Any amount of such finite configurations can then be placed disjointly around each other: The sum $x = \sum_{i \in \N} z_{k_i}$ is well defined and in $X$ if $k_i$ grows rapidly enough, since $X$ is an SFT. Further, with a sequence that grows rapidly enough, the evolutions of the summands $z_{k_i}$ of $x$ are disjoint in the sense that no two $z_{k_i}$ interact: we use the mortality of $z_{k_i}$ to assure it dies long before $z_{k_{i+1}}$ reaches it. But then the trace support of $x$ is the infinite union of the trace supports of the summands $z_{k_i}$, so $c$ is not asymptotically nilpotent.
\end{proof}

In particular, the previous lemma holds for the full shift. The generalization to subshifts with dense finite points is needed for the analogous generalization of Theorem~\ref{thm:dDCase} to Theorem~\ref{thm:FinitesDense}.

We note that simply having dense finite points and mortality of finite patterns is not enough in Lemma~\ref{lem:FiniteMortality}, even in the one-dimensional case. Simple examples of non-nilpotent CA on the full shift for which finite points are mortal are given in at least \cite{GaKuLe78,Ku03,KaGl12}.

\begin{theorem}
\label{thm:dDCase}
Asymptotic nilpotency implies nilpotency on $S^{\Z^d}$.
\end{theorem}

\begin{proof}
By Theorem~\ref{thm:1DCase}, we know that asymptotic nilpotency implies nilpotency in the case $d = 1$. We will prove the case $d = 2$, and informally explain how the general case is proved. So let $c : S^{\Z^2} \to S^{\Z^2}$ be asymptotically nilpotent with radius $r$. Let further $X = S^{\Z^2}$ and let $v^1 = (1, 0)$ and $v^2 = (0, 1)$ be the standard basis vectors of $\Z^2$. First, we note that $c$ is nilpotent on $X(p v^2)$ for all $p$: the one-dimensional CA $c_{2, p}$ is asymptotically nilpotent on the one-dimensional full shift $\phi_{2, p}(X(p v^2)) = (S^p)^\Z$, and thus nilpotent by Theorem~\ref{thm:1DCase}, implying that $c$ is nilpotent on $X(p v^2)$.

Now, the basis of the proof is the following observation: given any finite configuration, if we `add a vertical period', it becomes mortal. That is, let $x$ be a finite configuration, and let $\supp(x) \subset B_\ell((0, 0))$ for some $k$. Then, if $m \geq 2\ell + 1$, the configuration $\sum_{i \in \Z} \sigma^{imv^2}(x) \in S^{\Z^d}$ is well-defined, and it is mortal because it is vertically periodic, by the argument of the previous paragraph. Then, if there exist finite patterns extending arbitrarily far in the directions spanned by $v^1$, we can use argumentation similar to that of \cite{GuGa08} with `horizontally finite' points (that is, vertically periodic points which use only finitely many columns) to find a contradiction to asymptotic nilpotency.

Let us make this more precise. We first claim that there exists $k$ such that for all finite configurations $x$, we have
\begin{equation}
\label{eq:supps}
\bigcup_{i \in \N} \supp(c^i(x)) \subset \tower(2, k, \supp(x)).
\end{equation}
We show that if this does not hold, we can construct a configuration that contradicts asymptotic nilpotency, so assume that for all $k$, the finite configuration $x^k \in S^{\Z^2}$ is a counterexample to \eqref{eq:supps} for $k$, and for all $k$, let
\[ u^k \in (\bigcup_{i \in \N} \supp(c^i(x^k))) \setminus \tower(2, k, \supp(x^k)). \]

Using the hypothetical $x^k$ and $u^k$, we inductively construct a configuration $x$ where $c^j(x)_{(0, 0)} \neq 0$ for arbitrarily large $j$. For an illustration of the inductive step and what $y^1$, $y^2$ and $y^3$ might look like in what follows, see figures~\ref{fig:ProofA}, \ref{fig:ProofB}, \ref{fig:ProofC} and~\ref{fig:ProofD}. Let $y^0$ be the all zero configuration, and take as the induction hypothesis that
\begin{itemize}
\item $y^i$ is a mortal vertically periodic configuration,
\item $y^i$ has all of its nonzero cells in a finite amount of columns,
\item the trace support of $y^i$ is larger than that of $y^{i-1}$, if $i > 0$.
\end{itemize}
Now, let us construct $y^{i+1}$ assuming $y^i$ satisfies the induction hypothesis. Let $m$ be such that $y^i$ becomes zero in at most $m+1$ steps, and that $\supp(c^n(y^i)) \subset \tower(2, m, (0, 0))$ for all $n$. Such $m$ exists because $y^i$ is mortal and has its nonzero cells in finitely many columns. Then for $k = (r + 1)m + 2r + 1$, we have that for $y^i + \sigma^{u^k}(x^k)$, two nonzero cells arising from $y^i$ and $\sigma^{u^k}(x^k)$, respectively, are never seen in the same neighborhood, since the configuration $y^i$ dies before it is reached by the nonzero cells evolving from $\sigma^{u^k}(x^k)$. More precisely, we have 
\[ c^j(y^i + \sigma^{u^k}(x^k)) = c^j(y^i) + c^j(\sigma^{u^k}(x^k)), \]
for all $j \in \N$.

We thus see that, at the origin, the nonzero values arising from $y^i$ are followed by those arising from $\sigma^{u^k}(x^k)$ in the evolution of $c$, and by the assumption on $x^k$ and $u^k$, the trace support increases in cardinality by at least $1$. By adding a vertical period for $\sigma^{u^k}(x^k)$, obtaining a configuration $y$, we see that $y^{i+1} = y^i + y$ is mortal. Furthermore, if the period of $y$ is chosen large enough that the relevant initial part of the trace is not changed, the trace support of $y^{i+1}$ is a proper superset of that of $y^i$, and it thus satisfies the induction hypothesis. The configuration $x = \lim_i y^i$ now contradicts asymptotic nilpotency, which concludes the proof that for some $k$, \eqref{eq:supps} holds for all finite $x$.

This means that no finite configuration can extend arbitrarily far in the horizontal directions, and with an analogous proof we see that no finite configuration can extend arbitrarily far vertically either. It is now easy to see that every finite configuration is in fact mortal, and Lemma~\ref{lem:FiniteMortality} concludes the proof.

Now, consider the general case $d > 1$. We can prove this in two ways, either reducing directly to the case $d = 1$ or proceeding by induction on $d$. To reduce to the case $d - 1$, we take any basis vector $v$ and, analogously to the case $d = 2$, prove that a finite pattern can only expand arbitrarily far in the directions spanned by $v$, using our previous argument with the $y^i$ obtained from finite configurations by adding period $pv$, which are mortal by the induction hypothesis on $d - 1$, by using the homeomorphisms $\phi_{j, p}$ where $v = v_j$. Now, a finite configuration cannot expand arbitrarily far in any direction (since towers given by any two distinct base vectors have a finite intersection), and Lemma~\ref{lem:FiniteMortality} applies.

To reduce to the case $d = 1$ (Theorem~\ref{thm:1DCase}) directly, we can, for all basis vectors $v$, extend finite patterns into configurations with $d - 1$ (that is, all but $v$) directions of periodicity. Then, running $c$ on such a configuration simulates a one-dimensional asymptotically nilpotent cellular automaton, which is then nilpotent, and we can show using our previous argument that a finite pattern cannot extend arbitrarily far in the directions spanned by $v$. Now, going through all basis vectors $v$, we see that finite configurations are mortal, and Lemma~\ref{lem:FiniteMortality} again applies.
\end{proof}

\newcommand{\sq}[2]{
	\fill[draw=black,color=black] (#1,#2) rectangle (#1+1,#2+1);
}

\begin{figure}
\begin{center}
\begin{tikzpicture}[scale = 0.35]
\draw[fill,black!30] (0,-4) rectangle (1,12);
\draw[fill,black!30] (-4,0) rectangle (12,1);

\sq{6}{7};
\sq{5}{6};
\sq{5}{5};
\sq{6}{5};
\sq{7}{5};

\draw[black!50] (-4,-4) grid (12,12);

\node[draw, fill, circle, scale=0.3] (origin) at (0.5,0.5) {};
\draw[->, black] (5.5,5.5) -- (origin);
\end{tikzpicture}
\end{center}
\caption{A finite pattern $P_1$ which reaches the origin in finite time in the action of an asymptotically nilpotent CA $c$. We only use a Game of Life glider for illustrative purposes, as GoL is not asymptotically nilpotent.}
\label{fig:ProofA}
\end{figure}
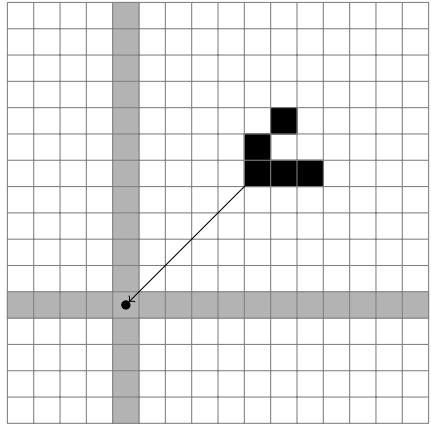

\begin{figure}
\begin{center}
\begin{tikzpicture}[scale = 0.14]
\draw[black!30,thick] (0.5,-17) -- (0.5,30);
\draw[black!30,thick] (-4,0.5) -- (12,0.5);

\foreach \i in {-18, -9, 0, 9, 18} {
	\sq{6}{7+\i};
	\sq{5}{6+\i};
	\sq{5}{5+\i};
	\sq{6}{5+\i};
	\sq{7}{5+\i};
}

\draw[white] (-4,-17) grid (12,30);
\draw (-4,-17) rectangle (12,30);

\node[draw, fill, circle, scale=0.3] (origin) at (0.5,0.5) {};
\draw[->, black] (5.5,5.5) -- (origin);
\end{tikzpicture}
\end{center}
\caption{The pattern $P_1$ of Fig.~\ref{fig:ProofA} with an added vertical period, resulting in a possible choice for the configuration $y^1$ in the proof of Theorem~\ref{thm:dDCase}. The 2D CA $c$ will simulate an asymptotically nilpotent 1D CA on this configuration, so the configuration must be mortal by Theorem~\ref{thm:1DCase}.}
\label{fig:ProofB}
\end{figure}
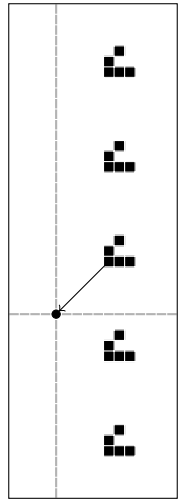

\begin{figure}
\begin{center}
\begin{tikzpicture}[scale = 0.14]
\draw[black!30,thick] (0.5,-17) -- (0.5,47);
\draw[black!30,thick] (-60,0.5) -- (12,0.5);

\foreach \i in {-18, -9, 0, 9, 18, 27, 36} {
	\sq{6}{7+\i};
	\sq{5}{6+\i};
	\sq{5}{5+\i};
	\sq{6}{5+\i};
	\sq{7}{5+\i};
}

\fill[draw=black,color=black] (-56,38) rectangle (-55,39);
\fill[draw=black,color=black] (-56,39) rectangle (-55,40);
\fill[draw=black,color=black] (-55,38) rectangle (-54,39);
\fill[draw=black,color=black] (-55,39) rectangle (-54,40);
\fill[draw=black,color=black] (-46,37) rectangle (-45,38);
\fill[draw=black,color=black] (-46,38) rectangle (-45,39);
\fill[draw=black,color=black] (-46,39) rectangle (-45,40);
\fill[draw=black,color=black] (-45,36) rectangle (-44,37);
\fill[draw=black,color=black] (-45,40) rectangle (-44,41);
\fill[draw=black,color=black] (-44,35) rectangle (-43,36);
\fill[draw=black,color=black] (-44,41) rectangle (-43,42);
\fill[draw=black,color=black] (-43,35) rectangle (-42,36);
\fill[draw=black,color=black] (-43,41) rectangle (-42,42);
\fill[draw=black,color=black] (-42,38) rectangle (-41,39);
\fill[draw=black,color=black] (-41,36) rectangle (-40,37);
\fill[draw=black,color=black] (-41,40) rectangle (-40,41);
\fill[draw=black,color=black] (-40,37) rectangle (-39,38);
\fill[draw=black,color=black] (-40,38) rectangle (-39,39);
\fill[draw=black,color=black] (-40,39) rectangle (-39,40);
\fill[draw=black,color=black] (-39,38) rectangle (-38,39);
\fill[draw=black,color=black] (-36,39) rectangle (-35,40);
\fill[draw=black,color=black] (-36,40) rectangle (-35,41);
\fill[draw=black,color=black] (-36,41) rectangle (-35,42);
\fill[draw=black,color=black] (-35,39) rectangle (-34,40);
\fill[draw=black,color=black] (-35,40) rectangle (-34,41);
\fill[draw=black,color=black] (-35,41) rectangle (-34,42);
\fill[draw=black,color=black] (-34,38) rectangle (-33,39);
\fill[draw=black,color=black] (-34,42) rectangle (-33,43);
\fill[draw=black,color=black] (-32,37) rectangle (-31,38);
\fill[draw=black,color=black] (-32,38) rectangle (-31,39);
\fill[draw=black,color=black] (-32,42) rectangle (-31,43);
\fill[draw=black,color=black] (-32,43) rectangle (-31,44);
\fill[draw=black,color=black] (-22,40) rectangle (-21,41);
\fill[draw=black,color=black] (-22,41) rectangle (-21,42);
\fill[draw=black,color=black] (-21,40) rectangle (-20,41);
\fill[draw=black,color=black] (-21,41) rectangle (-20,42);

\draw[white] (-60,-17) grid (12,47);
\draw (-60,-17) rectangle (12,47);

\node[draw, fill, circle, scale=0.3] (origin) at (0.5,0.5) {};
\draw[->, black] (5.5,5.5) -- (origin);
\draw[->, black] (-38.5,38.5) -- (origin);
\end{tikzpicture}
\end{center}
\caption{The mortal configuration $y^1$ of Fig.~\ref{fig:ProofB} with a new finite pattern $P_2$ added on the left. The pattern $P_2$ eventually reaches the origin in the action of $c$, but it is far enough away that the copies of $P_1$ die out before $P_2$ reaches them.}
\label{fig:ProofC}
\end{figure}
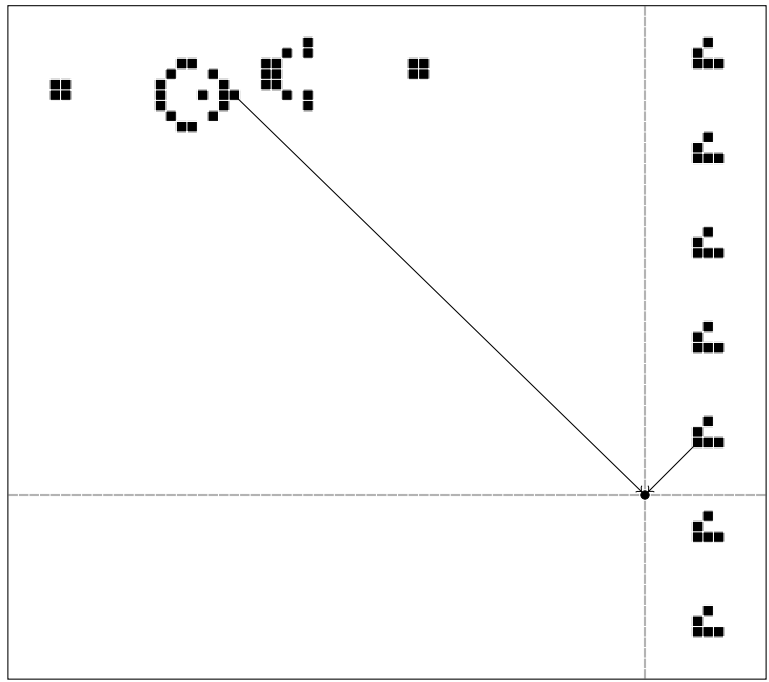

\begin{figure}
\begin{center}
\begin{tikzpicture}[scale = 0.1]
\foreach \i in {-36, -27, -18, -9, 0, 9, 18, 27, 36} {
	\sq{6}{7+\i};
	\sq{5}{6+\i};
	\sq{5}{5+\i};
	\sq{6}{5+\i};
	\sq{7}{5+\i};
}

\foreach \i in {0,-30,-60} {
	\fill[draw=black,color=black] (-56,38+\i) rectangle (-55,39+\i);
	\fill[draw=black,color=black] (-56,39+\i) rectangle (-55,40+\i);
	\fill[draw=black,color=black] (-55,38+\i) rectangle (-54,39+\i);
	\fill[draw=black,color=black] (-55,39+\i) rectangle (-54,40+\i);
	\fill[draw=black,color=black] (-46,37+\i) rectangle (-45,38+\i);
	\fill[draw=black,color=black] (-46,38+\i) rectangle (-45,39+\i);
	\fill[draw=black,color=black] (-46,39+\i) rectangle (-45,40+\i);
	\fill[draw=black,color=black] (-45,36+\i) rectangle (-44,37+\i);
	\fill[draw=black,color=black] (-45,40+\i) rectangle (-44,41+\i);
	\fill[draw=black,color=black] (-44,35+\i) rectangle (-43,36+\i);
	\fill[draw=black,color=black] (-44,41+\i) rectangle (-43,42+\i);
	\fill[draw=black,color=black] (-43,35+\i) rectangle (-42,36+\i);
	\fill[draw=black,color=black] (-43,41+\i) rectangle (-42,42+\i);
	\fill[draw=black,color=black] (-42,38+\i) rectangle (-41,39+\i);
	\fill[draw=black,color=black] (-41,36+\i) rectangle (-40,37+\i);
	\fill[draw=black,color=black] (-41,40+\i) rectangle (-40,41+\i);
	\fill[draw=black,color=black] (-40,37+\i) rectangle (-39,38+\i);
	\fill[draw=black,color=black] (-40,38+\i) rectangle (-39,39+\i);
	\fill[draw=black,color=black] (-40,39+\i) rectangle (-39,40+\i);
	\fill[draw=black,color=black] (-39,38+\i) rectangle (-38,39+\i);
	\fill[draw=black,color=black] (-36,39+\i) rectangle (-35,40+\i);
	\fill[draw=black,color=black] (-36,40+\i) rectangle (-35,41+\i);
	\fill[draw=black,color=black] (-36,41+\i) rectangle (-35,42+\i);
	\fill[draw=black,color=black] (-35,39+\i) rectangle (-34,40+\i);
	\fill[draw=black,color=black] (-35,40+\i) rectangle (-34,41+\i);
	\fill[draw=black,color=black] (-35,41+\i) rectangle (-34,42+\i);
	\fill[draw=black,color=black] (-34,38+\i) rectangle (-33,39+\i);
	\fill[draw=black,color=black] (-34,42+\i) rectangle (-33,43+\i);
	\fill[draw=black,color=black] (-32,37+\i) rectangle (-31,38+\i);
	\fill[draw=black,color=black] (-32,38+\i) rectangle (-31,39+\i);
	\fill[draw=black,color=black] (-32,42+\i) rectangle (-31,43+\i);
	\fill[draw=black,color=black] (-32,43+\i) rectangle (-31,44+\i);
	\fill[draw=black,color=black] (-22,40+\i) rectangle (-21,41+\i);
	\fill[draw=black,color=black] (-22,41+\i) rectangle (-21,42+\i);
	\fill[draw=black,color=black] (-21,40+\i) rectangle (-20,41+\i);
	\fill[draw=black,color=black] (-21,41+\i) rectangle (-20,42+\i);
}

\fill[draw=black,color=black] (81,-2) rectangle (82,-1);
\fill[draw=black,color=black] (81,0) rectangle (82,1);
\fill[draw=black,color=black] (80,1) rectangle (81,2);
\fill[draw=black,color=black] (79,1) rectangle (80,2);
\fill[draw=black,color=black] (78,1) rectangle (79,2);
\fill[draw=black,color=black] (77,1) rectangle (78,2);
\fill[draw=black,color=black] (77,0) rectangle (78,1);
\fill[draw=black,color=black] (77,-1) rectangle (78,0);
\fill[draw=black,color=black] (78,-2) rectangle (79,-1);

\draw[white] (-60,-35) grid (86,47);
\draw (-60,-35) rectangle (86,47);

\node[draw, fill, circle, scale=0.3] (origin) at (0.5,0.5) {};
\draw[->, black] (5.5,5.5) -- (origin);
\draw[->, black] (-38.5,38.5) -- (origin);
\draw[->, black] (77.5,0.5) -- (origin);
\end{tikzpicture}
\end{center}
\caption{The configuration of Fig.~\ref{fig:ProofC} with an added vertical period ($y^2$ in the proof of Theorem~\ref{thm:dDCase}) together with a new finite pattern $P_3$ on the right. Again, $P_3$ reaches the origin, but sufficiently late that the rest of the configuration has died out before it does. The configuration $y^3$ would be obtained by once again adding a vertical period.}
\label{fig:ProofD}
\end{figure}
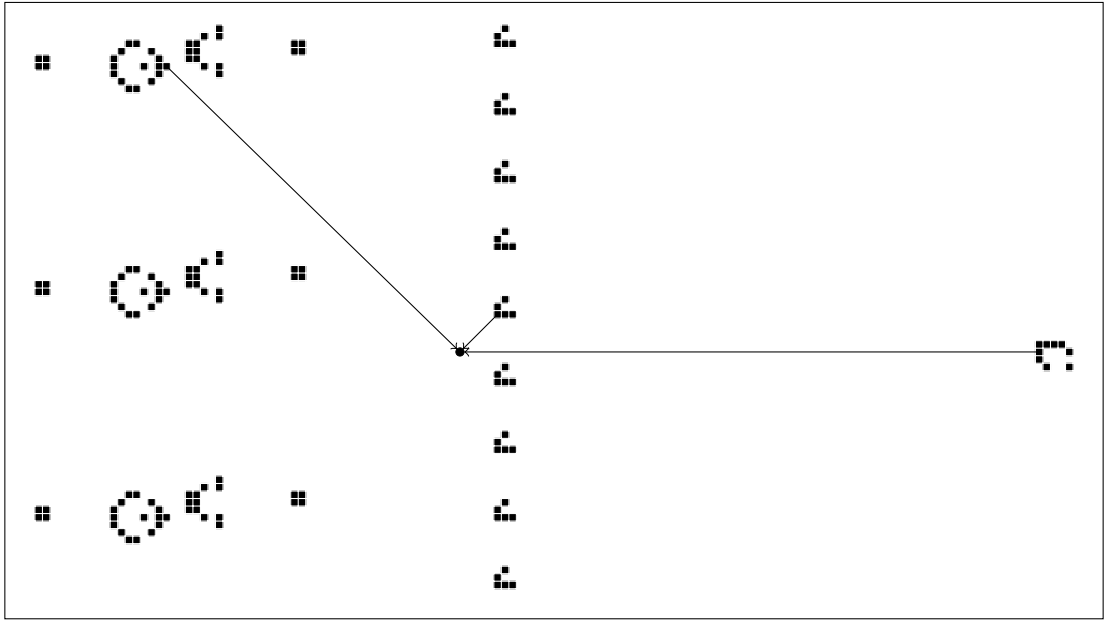

It is hard to imagine a group in which asymptotic nilpotency does not imply nilpotency, leading to the obvious question:

\begin{question}
Does asymptotic nilpotency imply nilpotency on full shifts in all finitely generated groups?
\end{question}

When the full shift is replaced by an arbitrary subshift, this is easily seen not to be the case. For example, asymptotic nilpotency does not imply nilpotency on all one-dimensional sofic shifts: a nontrivial shift action is asymptotically nilpotent but not nilpotent on $\lang^{-1}(0^*10^*)$. In fact, asymptotic nilpotency does not even necessarily imply nilpotency if the sofic shift is mixing:

\begin{example}
Let $X$ be the mixing sofic shift $\lang^{-1}((0^*l0^*r)^*)$ and let $c : X \to X$ be the cellular automaton that moves $l$ to the left and $r$ to the right, removing $l$ and $r$ when they collide. Clearly, $c$ is asymptotically nilpotent but not nilpotent.
\end{example}

We can, however, prove that asymptotic nilpotency implies nilpotency on all SFTs in dimension one, as a corollary of a theorem in \cite{GuGa10}.

\begin{theorem}[Theorem 4 of \cite{GuGa10}]
\label{thm:TransitiveSFTCase}
Asymptotic nilpotency implies nilpotency on transitive $1$-dimensional SFTs.
\end{theorem}

\begin{corollary}
\label{cor:SFTCase}
Asymptotic nilpotency implies nilpotency on $1$-dimensional SFTs.
\end{corollary}

\begin{proof}
Let $X \subset S^\Z$ be an SFT and let $c : X \to X$ be asymptotically nilpotent. On each transitive component $Y$ of $X$, $c$ is in fact nilpotent by Theorem~\ref{thm:TransitiveSFTCase}, and it is easy to see that the $0$-symbol must be the same for all transitive components. Let $n$ be such that all the finitely many transitive components of $X$ are mapped to $0^\Z$ by $c^n$. Now, consider an arbitrary configuration $x \in X$. The configuration $c^n(x) \in X$ must be both left and right asymptotic to $0^\Z$, that is, $c^n(x)_i = 0$ if $|i|$ is large enough, so by the assumption that $X$ is an SFT, $c^n(x) \in Y$ for some transitive component $Y$ of $X$. But this means $c^{2n}(x) = 0^\Z$.
\end{proof}

Of course, sofic shifts where asymptotic nilpotency does not imply nilpotency exist in any dimension, since they exist in dimension one. However, the case of an SFT seems harder due to the surprisingly complicated nature of multidimensional SFTs. At least if finite configurations are dense in a $d$-dimensional SFT, the proof of Theorem~\ref{thm:dDCase} works rather directly using Corollary~\ref{cor:SFTCase} and Lemma~\ref{lem:FiniteMortality}.

\begin{theorem}
\label{thm:FinitesDense}
If $X \subset S^{\Z^d}$ is an SFT where finite points are dense and $c : X \to X$ is an asymptotically nilpotent cellular automaton, then $c$ is nilpotent.
\end{theorem}

\begin{proof}
The case $d = 1$ follows from Corollary~\ref{cor:SFTCase}. We again only explicitly consider the case $d = 2$, and the proof for $d > 2$ is obtained as in Theorem~\ref{thm:dDCase}. First, we note that the subshift $X(p v^2)$ is mapped through $\phi_{2, p}$ into a $1$-dimensional SFT, so Corollary~\ref{cor:SFTCase} applies. Now, if $x \in X$ is a finite point, the sum $\sum_{i \in \Z} \sigma^{imv^2}(x)$ is defined for all large enough $m$, since $X$ is an SFT. This means that we can prove, using the arguments of the proof of Theorem~\ref{thm:dDCase}, that finite points of $X$ cannot expand arbitrarily far horizontally or vertically, which implies they are mortal. Then, Lemma~\ref{lem:FiniteMortality} implies that $c$ is nilpotent.
\end{proof}

We do not know much about the class of SFTs with dense finite points in dimension $d > 1$. Note that for example, unlike in the case $d = 1$, not every $2$-dimensional transitive SFT with a uniform configuration has this property, $\{x \in \{0, 1\}^{\Z^2} \;|\; x = \sigma^{v^1}(x)\}$ being a trivial example.

\begin{question}
Does asymptotic nilpotency imply nilpotency on all SFTs?
\end{question}

More generally, it would be interesting to see what could be done in the more general framework of projective subdynamics, where the cellular automata are -- in a sense -- nondeterministic. For instance, it would be interesting to see whether techniques similar to ours would help in extending Theorem 6.4 of \cite{PaSc10} to dimensions higher than $2$ at least in some natural subcases.

\section*{Acknowledgements}

I would like to thank Ilkka T\"orm\"a for his help with (naturally unsuccessfully) trying to find a counterexample to Theorem~\ref{thm:dDCase}, Jarkko Kari for the proof of Proposition~\ref{prop:WeakIsStrong}, and Pierre Guillon for shedding light on the relevant generalizations of this problem, and on other possible uses of this work. And of course, I would like to thank all three and the anonymous referees of AUTOMATA \& JAC~2012 for their comments on this article. The pattern $P_2$ in Figure~\ref{fig:ProofC} is due to Bill Gosper \cite{WPGun12}. The patterns $P_1$ and $P_3$ are attributed to Rikhard K. Guy and John Conway respectively [folklore].


\def\ocirc#1{\ifmmode\setbox0=\hbox{$#1$}\dimen0=\ht0\advance\dimen0
  by1pt\rlap{\hbox
  to\wd0{\hss\raise\dimen0\hbox{\hskip.2em$\scriptscriptstyle\circ$}\hss}}#1\e%
lse{\accent"17 #1}\fi}

\end{document}